[1994/12/01]
\documentclass[a4paper]{amsart}
\usepackage{amssymb,amsxtra}
\usepackage{amsmath,amsthm,amscd}

\chardef\bslash=`\\ 

\newtheorem{thm}{Theorem}[section]
\newtheorem{cor}[thm]{Corollary}
\newtheorem{lem}[thm]{Lemma}
\newtheorem{prop}[thm]{Proposition}

\theoremstyle{definition}
\newtheorem{defn}{Definition}[section]

\theoremstyle{remark}

\newcommand{\thmref}[1]{Theorem~\ref{#1}}

\newcommand{\lemref}[1]{Lemma~\ref{#1}}
\newcommand{\propref}[1]{Proposition~\ref{#1}}
\newcommand{\corref}[1]{Corollary~\ref{#1}}

\newcommand{\eval}[2][\right]{\relax
  \ifx#1\right\relax \left.\fi#2#1\rvert}

\def\rsd{\mbox{$\mathop{\mathrel\times\joinrel\mathrel{\vrule height 5pt
 depth 0pt}} $}}

\def\lsd{\mbox{$\mathop{\mathrel{\vrule height 5pt
 depth 0pt}\joinrel\mathrel\times} $}}

\title{CARTAN SUBALGEBRAS IN C$^*$-ALGEBRAS}
\author{Jean Renault}
\address{D\'epartment de Math\'ematiques, Universit\'e d'Orl\'eans,
45067 Orl\'eans, France}
\email{Jean.Renault@univ-orleans.fr}

\begin{document}
\vskip5mm
\begin{abstract}
According to J.~Feldman and C.~Moore's well-known theorem on Cartan subalgebras, a variant of the group
measure space construction gives an equivalence of categories between twisted countable
standard measured equivalence relations and Cartan pairs, i.e. a von Neumann algebra
(on a separable Hilbert space) together with a Cartan subalgebra. A.~Kumjian gave a
C$^*$-algebraic analogue of this theorem in the early eighties. After a short survey of
maximal abelian self-adjoint subalgebras in operator algebras, I present a
natural definition of a Cartan subalgebra in a C$^*$-algebra and an
extension of Kumjian's theorem which covers graph algebras and some foliation algebras.
\end{abstract} 

\maketitle
\markboth{Renault}
{Cartan subalgebras}

\renewcommand{\sectionmark}[1]{}

\section{Introduction}

One of the most fundamental constructions in the theory of operator algebras, namely the
crossed product construction, provides a subalgebra, i.e. a pair $(B,A)$ consisting of
an operator algebra $A$ and a subalgebra $B\subset A$, where $B$ is the original
algebra. The inclusion $B\subset A$ encodes the symmetries of the original dynamical
system. An obvious and naive question is to ask whether a given subalgebra arises from
some crossed product construction. From the very construction of the crossed product, a
necessary condition is that $B$ is regular in $A$, which means that $A$ is generated by
the normalizer of $B$. In the case of a crossed product by a group, duality theory
provides an answer (see Landstad
\cite{lan:duality}) which requires an external information, namely the dual action. Our
question is more in line with subfactor theory, where one extracts an algebraic object
(such as a paragroup or a quantum groupoid) solely from an inclusion of factors. Under
the assumption that $B$ is maximal abelian, the problem is somewhat more tractable. The
most satisfactory result in this direction is the Feldman-Moore theorem
\cite[Theorem 1]{fm:relations II}, which characterizes the subalgebras arising from the
construction of the von Neumann algebra of a measured countable equivalence relation.
These subalgebras are precisely the Cartan subalgebras, a nice kind of maximal abelian
self-adjoint subalgebras ({\it masas}) introduced previously by Vershik in
\cite{ver:cartan}: they are regular and there exists a faithful normal conditional
expectation of $A$ onto
$B$. The Cartan subalgebra contains exactly the same information as the equivalence
relation. This theorem leaves pending a number of interesting and difficult questions.
For example, the existence or the uniqueness of Cartan subalgebras in a given von
Neumann algebra. Another question is to determine if the equivalence relation arises
from a free action of a countable group and if one can expect uniqueness of the group.
Let us just say that there have been some recent breakthroughs on this question (Popa et alii). 

It was then natural to find a counterpart of the Feldman-Moore theorem for
C$^*$-algebras. In \cite{kum:diagonals}, Kumjian introduced  the notion of
a C$^*$-diagonal as the C$^*$-algebraic counterpart of a Cartan subalgebra and showed
that, via the groupoid
algebra construction, they correspond exactly to twisted \'etale equivalence relations.
A key ingredient of his theorem is his definition of the normalizer of a subalgebra (a
definition in terms of unitaries or partial isometries would be too restrictive). His
fundamental result, however, does not cover a number of important examples. For
example, Cuntz algebras, and more generally graph algebras, have obvious regular
{\it masas} which are not C$^*$-diagonals. The same is true for foliations algebras (or
rather their reduction to a full transversal). The reason is that the groupoids from
which they are constructed are essentially principal but not principal: they have some
isotropy that cannot be eliminated. It seems that, in the topological context,
essentially principal groupoids are more natural than principal groupoids (equivalence
relations). They are exactly the groupoids of germs of pseudogroups. Groupoids of germs
of pseudogroups present a technical difficulty: they may fail to be Hausdorff (they are
Hausdorff if and only if the pseudogroup is quasi-analytical). For the sake of
simplicity, our discussion will be limited to the Hausdorff case. We refer the
interested reader to a forthcoming paper about the non-Hausdorff case. A natural
definition of a Cartan subalgebra in the C$^*$-algebraic context is that it is a
{\it masa} which is regular and which admits a faithful conditional expectation. We show
that in the reduced C$^*$-algebra of an essentially principal Hausdorff \'etale groupoid
(endowed with a twist), the subalgebra corresponding to the unit space is a Cartan
subalgebra. Conversely, every Cartan subalgebra (if it exists!) arises in that fashion and completely
determines the groupoid and the twist. Our proof closely follows Kumjian's. The comparison with Kumjian's theorem shows that a Cartan subalgebra has the unique extension
property if and only if the corresponding groupoid is principal. As a corollary of
the main result, we obtain that a Cartan subalgebra has a unique conditional
expectation, which is clear when the subalgebra has the unique extension property but
not so in the general case.

Here is a brief description of the content of this paper. In Section 2, I will review
some basic facts about {\it masas} in von Neumann algebras, the Feldman-Moore
theorem and some more recent results on Cartan subalgebras. In Section 3, I will review the characterization of essentially principal groupoids as groupoids of germs of pseudogroups of local homeomorphisms.  In Section 4, I will review  the construction of the reduced C$^*$-algebra of a locally compact Hausdorff groupoid $G$ with Haar system and endowed with a twist. I will show that, when $G$ is \'etale, the subalgebra of the unit space is a {\it masa} if and only if $G$ is essentially principal.
In fact, this is what we call a Cartan subalgebra in the C$^*$-algebraic context: it
means a {\it masa} which is regular and which has a faithful conditional expectation.
In Section 5, we show the converse: every Cartan subalgebra arises from an
essentially principal
\'etale groupoid endowed with a twist. This groupoid together with its twist is a
complete isomorphism invariant of the Cartan subalgebra. We end with examples of Cartan subalgebras in C$^*$-algebras. 

This paper is a written version of a talk given at OPAW2006 in Belfast. I heartily thank the organizers, M.~Mathieu and I.~Todorov, for the invitation and the participants, in particular P.~Resende, for stimulating discussions. I also thank A.~Kumjian and I.~Moerdijk for their interest and their help. 

\section{Cartan subalgebras in von Neumann algebras}

The basic example of a {\it masa} in an operator algebra is the subalgebra $D_n$ of
diagonal matrices in the algebra $M_n$ of complex-valued $(n,n)$-matrices. Every
{\it masa} in
$M_n$ is conjugated to it by a unitary (this is essentially the well-known result
that every normal complex matrix admits an orthonormal basis of eigenvectors). The
problem at hand is to find suitable generalizations of this basic example.

The most immediate generalization is to replace ${\bf C}^n$ by an infinite dimensional
separable Hilbert space
$H$ and $M_n$ by the von Neumann algebra ${\mathcal B}(H)$ of
all bounded linear operators on
$H$. The spectral theorem tells us that, up to conjugation by a unitary, {\it masas} in
${\mathcal B}(H)$ are of the form
$L^\infty(X)$, acting by multiplication on
$H=L^2(X)$, where $X$ is an infinite standard measure space. Usually, one distinguishes
the case of $X=[0,1]$ endowed with Lebesgue measure and the case of $X={\bf N}$  endowed
with counting measure. In the first case, the {\it masa} is called diffuse and in the
second case, it is called atomic. Atomic {\it masas} $A$ in ${\mathcal B}(H)$ can be
characterized by the existence of a normal conditional expectation $P: {\mathcal
B}(H)\rightarrow A$. Indeed, when $H=\ell^2({\bf N})$, operators are given by matrices
and $P$ is the restriction to the diagonal. What we are looking for is precisely a
generalization of these atomic {\it masas}.

There is no complete classification of {\it masas} in non-type I factors. In fact, the
study of {\it masas} in non-type I factors looks like a rather formidable
task. In 1954, J.~Dixmier \cite{dix:masa} discovered the existence of non-regular {\it masas}.
A {\it masa}
$A$ in a von Neumann algebra $M$ is called {\it regular} if its normalizer $N(A)$ (the
group of unitaries $u$ in $M$ which normalize $A$, in the sense that $uAu^*=A$)
generates $M$ as a von Neumann algebra. On the other hand, it is called {\it singular}
if $N(A)$ is contained in $A$. When $N(A)$ acts ergodically on $A$, the {\it masa} $A$ is
called {\it semi-regular}. Every {\it masa} in ${\mathcal B}(H)$ (or in a type $I$ von Neumann
algebra) is regular. Dixmier gave an example of a singular {\it masa} in the hyperfinite
$II_1$ factor (check if it is the image of a normal conditional expectation). Thus, in
order to generalize the atomic {\it masas} of ${\mathcal B}(H)$, we shall consider {\it masas}
which are regular and the image of a normal conditional expectation:

\begin{defn} (Vershik\cite{ver:cartan}, Feldman-Moore\cite[Definition 3.1]{fm:relations
II}) An abelian subalgebra $A$ of a von Neumann algebra $M$ is called a {\it Cartan
subalgebra} if
\begin{enumerate}
\item $A$ is a {\it masa};
\item $A$ is regular;
\item there exists a faithful normal conditional expectation of $M$ onto $A$. 
\end{enumerate}
\end{defn}

Cartan subalgebras are intimately related to ergodic theory. Indeed, if $M$ arises by
the classical group measure construction from a free action of a discrete countable
group $\Gamma$ on a measure space $(X,\mu)$, then $L^\infty(X,\mu)$ is naturally
imbedded in $M$ as a Cartan subalgebra (\cite{mvn:rings}). Following generalizations by
G.~Zeller-Meier \cite[Remarque 8.11]{zm:croises}, W.~Krieger \cite{kri:type III} and
P.~Hahn
\cite{hah:regular}, J.~Feldman and C.~Moore give in \cite{fm:relations II} the most
direct construction of Cartan subalgebras. It relies on the notion of a countable
standard measured equivalence relation. Here is its definition: $(X,{\mathcal B},\mu)$
is a standard measured space and
$R$ is an equivalence relation on $X$ such that
its classes are countable, its graph $R$ is a Borel subset of $X\times X$ and the
measure $\mu$ is quasi-invariant under $R$. The last condition means that the measures
$r^*\mu$ and $s^*\mu$ on $R$ are equivalent (where
$r,s$ denote respectively the first and the second projections of $R$ onto $X$ and
$r^*\mu(f)=\int\sum_yf(x,y)d\mu(x)$ for a positive Borel function $f$ on $R$). The orbit
equivalence relation of an action of a discrete countable group $\Gamma$ on a measure
space $(X,\mu)$ preserving the measure class of $\mu$ is an example (in fact, according
to \cite[Theorem 1]{fm:relations I}, it is the most general example) of a countable
standard equivalence relation. The construction of the von Neumann algebra $M=W^*(R)$
mimicks the construction of the algebra of matrices $M_n$. Its elements are complex
Borel functions on $R$, the product is matrix multiplication and involution is the
usual matrix conjugation. Of course, in order to have an involutive algebra of bounded
operators, some conditions are required on these functions: they act by left
multiplication as operators on
$L^2(R,s^*\mu)$ and we ask these operators to be bounded. The subalgebra $A$ of
diagonal matrices (functions supported on the diagonal of $R$), which is isomorphic to
$L^\infty(X,\mu)$, is a Cartan subalgebra of
$M$. When $X={\bf N}$ and $\mu$ is the counting measure, one retrieves the atomic {\it masa}
of ${\mathcal B}(\ell^2({\bf N}))$. This construction can be twisted by a 2-cocycle
$\sigma\in Z^2(R,{\bf T})$; explicitly, $\sigma$ is a Borel function on
$R^{(2)}=\{(x,y,z)\in X\times X\times X: (x,y),(y,z)\in R\}$ with values in the group
of complex numbers of module 1 such that
$\sigma(x,y,z)\sigma(x,z,t)=\sigma(x,y,t)\sigma(y,z,t)$. The only modification is to
define as product the twisted matrix multiplication $f*g(x,z)=\sum
f(x,y)g(y,z)\sigma(x,y,z)$. This yields the von Neumann algebra $M=W^*(R,\sigma)$ and
its Cartan subalgebra $A=L^\infty(X,\mu)$ of diagonal matrices. The Feldman-Moore
theorem gives the converse.

\begin{thm}\cite[Theorem 1]{fm:relations II} Let $A$ be a Cartan subalgebra of a von
Neumann algebra
$M$ on a separable Hilbert space. Then there exists a countable standard
measured equivalence relation $R$ on $(X,\mu)$, a $\sigma\in Z^2(R,{\bf T})$ and an
isomorphism of $M$ onto $W^*(R,\sigma)$ carrying $A$ onto the diagonal subalgebra
$L^\infty(X,\mu)$. The twisted relation $(R,\sigma)$ is unique up to isomorphism.
\end{thm}

The main lines of the proof will be found in the C$^*$-algebraic
version of this result. This theorem completely elucidates the
structure of Cartan subalgebras. It says nothing about the existence and the uniqueness of
Cartan subalgebras in a given von Neumann algebra. We have seen that in
${\mathcal B}(H)$ itself, there exists a Cartan subalgebra, which is unique up to conjugacy.
The same result holds in every injective von Neumann algebra. More precisely, two
Cartan subalgebras of an injective von Neumann algebra are always conjugate by an
automorphism (but not always inner conjugate, as observed in \cite{fm:relations II}).
This important uniqueness result appears as 
\cite[Corollary 11]{cfw:amenable}. W.~Krieger had previously shown in \cite[Theorem 8.4]{kri:isomorphism} that two Cartan subalgebras of a von Neumann algebra $M$ which
produce hyperfinite (\cite[Definition 4.1]{fm:relations I}) equivalence relations are
conjugate (then, $M$ is necessarily hyperfinite). On the other hand, it is not 
difficult to show that a Cartan subalgebra of an injective von Neumann algebra produces
an amenable (\cite[Definition 6]{cfw:amenable}) equivalence
relation. Since Connes-Feldman-Weiss's theorem  states that an equivalence relation is
amenable if and only if it is hyperfinite, Krieger's uniqueness theorem can be applied.
The general situation is more
complex. Here are some results related to Cartan subalgebras of type
$II_1$ factors. In
\cite{cj:two}, A.~Connes and V.~Jones give an example of a $II_1$ factor with at least
two non-conjugate Cartan subalgebras. Then S.~Popa constructs in \cite{pop:rigidity} a
$II_1$ factor with uncountably many non-conjugate Cartan subalgebras. These examples
use Kazhdan's property $T$. In \cite{voi:cartan}, D.~Voiculescu shows that for
$n\ge 2$, the von Neumann algebra $L({\bf F}_n)$ of the free group ${\bf F}_n$ on $n$
generators  has no Cartan subalgebra.  Despite these rather negative results, it
seems that the notion of Cartan subalgebra still has a r\^ole to play in the theory of
$II_1$ factors. For example, S.~Popa has recently (see \cite{pop:rigidity II})
constructed and studied a large class of type $II_1$ factors (from Bernoulli
actions of groups with property $T$) which have a distinguished Cartan
subalgebra, unique up to inner conjugacy. Still more recently N.~Ozawa and
S.~Popa give in
\cite{op:cartan} on one hand many examples of $II_1$ factors which do not have any Cartan
subalgebra and on the other hand a new class of $II_1$ factors which have a
unique Cartan subalgebra, in fact unique not only up to conjugacy but to inner conjugacy. This class
consists of all the profinite ergodic probability preserving actions of free
groups ${\bf F}_n$ with $n\ge 2$.

\section{Essentially principal groupoids.}

The purpose of this section is mainly notational. It recalls elementary facts about \'etale groupoids and pseudogroups of homeomorphisms. Concerning groupoids, we shall use the notation of \cite{dr:amenable}. Other
relevant references are \cite{ren:approach} and \cite{pat:groupoids}.
Given a groupoid $G$, $G^{(0)}$ will
denote its unit space and $G^{(2)}$ the set of composable pairs. Usually,
elements of $G$ will be denoted by Greek letters as $\gamma$ and
elements of $G^{(0)}$ by Roman letters as $x,y$. The range and
source maps from $G$ to $G^{(0)}$ will be denoted respectively by $r$ and
$s$. The fibers of the range and source maps are denoted respectively
$G^x=r^{-1}(x)$ and $G_y=s^{-1}(y)$. The inverse map $G\rightarrow G$ is written
$\gamma\mapsto\gamma^{-1}$, the inclusion map $G^{(0)}\rightarrow G$ is
written $x\mapsto x$ and the product map $G^{(2)}\rightarrow G$ is
written $(\gamma,\gamma')\mapsto \gamma\gamma'$. 
 The {\it isotropy bundle} is $G'=\{\gamma\in
G:r(\gamma)=s(\gamma)\}$. 

In the topological setting, we
assume that the groupoid $G$ is a topological space and that the structure
maps are continuous, where $G^{(2)}$ has the topology induced by $G\times
G$ and $G^{(0)}$ has the  topology induced by $G$. We assume furthermore
that the range and source maps are surjective and open.  A topological groupoid
$G$ is called {\it \'etale} when its range or source maps are local homeomorphisms
from $G$ onto $G^{(0)}$. 

Recall that a subset $A$ of a groupoid
$G$ is called an {\it $r$-section} [resp. a {\it $s$-section}] if the restriction of
$r$ [resp. $s$] to $A$ is injective. A {\it bisection} is a subset $S\subset G$ which is
both an $r$-section and a $s$-section. If
$G$ is an  \'etale topological groupoid, it has a cover of open bisections. A
bisection $S$ defines a map $\alpha_S: s(S)\rightarrow r(S)$ such that
$\alpha_S(x)=r(Sx)$ for $x\in s(S)$. If moreover $G$ is \'etale and
$S$ is an open bisection, this map is a
homeomorphism. The open bisections of an \'etale groupoid $G$ form an {\it
inverse semigroup} ${\mathcal S}={\mathcal S}(G)$: the composition law is
$$ST=\{\gamma\gamma': (\gamma,\gamma')\in (S\times T)\cap G^{(2)}\}$$ and the
inverse of $S$ is the image of $S$ by the inverse map. The inverse semigroup
relations, which are $(RS)T=R(ST)$, $(ST)^{-1}=T^{-1}S^{-1}$ and $SS^{-1}S=S$,
are indeed satisfied. 

Our second main example of inverse semigroup is the set
${\mathcal P}art(X)$ of all partial  homeomorphisms of a topological space
$X$ endowed with composition and inverse. By partial homeomorphism, we mean a
homeomorphism $\varphi: U\rightarrow V$, where $U,V$ are open subsets of $X$. More
generally, one defines a {\it pseudogroup} on a topological space $X$ as a
sub-inverse semigroup of ${\mathcal P}art(X)$.  We say that the pseudogroup $\mathcal G$ is {\it
ample} if every partial homeomorphism $\varphi$ which locally belongs to $\mathcal
G$ (i.e. every point in the domain of $\varphi$ has an open neighborhood $U$ such
that
$\varphi_{|U}=\beta_{|U}$ with $\beta\in
\mathcal G$) does belong to $\mathcal G$. Given a pseudogroup $\mathcal G$, we denote by
$[{\mathcal G}]$ the set of partial homeomorphisms which belong locally  to
$\mathcal G$; it is an ample pseudogroup called the ample pseudogroup of $\mathcal
G$. Given a pseudogroup $\mathcal G$ on the topological space $X$, its groupoid of
germs is
$$G=\{[x,\varphi,y],\quad \varphi\in {\mathcal G},  y\in dom(\varphi),
x=\varphi(y)\}.$$
where $[x,\varphi,y]=[x,\psi,y]$ iff $\varphi$ and $\psi$ have the same germ at $y$,
i.e. there exists a neighborhood $V$ of $y$ in $X$ such that
$\varphi_{|V}=\psi_{|V}$. Its groupoid structure is defined by the range and source
maps $r[x,\varphi,y]=x, s[x,\varphi,y]=y$, the product
$[x,\varphi,y][y,\psi,z]=[x,\varphi\psi,z]$ and the inverse
$[x,\varphi,y]^{-1}=[y,\varphi^{-1},x]$. Its topology is the topology of germs, defined
by the basic open sets 
$${\mathcal U}(U,\varphi,V)=\{[x,\varphi,y]\in G: x\in U, y\in V\}$$ 
where $U,V$ are open
subsets of $X$ and $\varphi\in {\mathcal G}$. Observe that the groupoid of germs $G$ of
the pseudogroup $\mathcal G$ on $X$ depends on the ample pseudogroup $[\mathcal G]$
only. 

Conversely, an \'etale groupoid $G$ defines a pseudogroup $\mathcal G$ on
$X=G^{(0)}$. Indeed, the map
$\alpha:S\mapsto
\alpha_S$ is an inverse semigroup homomorphism of the inverse semigroup of open
bisections 
${\mathcal S}$ into ${\mathcal P}art(X)$ which we call the canonical action
of ${\mathcal S}$ on $X$. The relevant pseudogroup is its range ${\mathcal
G}=\alpha({\mathcal S})$.

\begin{defn} Let us a say that an \'etale groupoid $G$ is
\begin{enumerate}
\item {\it principal} if $G'=G^{(0)}$
\item{\it essentially principal} if the interior of $G'$ is $G^{(0)}$. 
\end{enumerate}
\end{defn}

The following proposition gives a more common definition of an essentially principal groupoid. Let us warn the reader that this definition is not the same as Definition II.4.3 of \cite{ren:approach}, which requires that the property holds not only for $G$ but also for all reductions $G_{|F}$, where $F$ is a closed invariant subset of $G^{(0)}$.

\begin{prop} Let us assume that $G$ is a second countable Hausdorff \'etale
groupoid and that its unit space $G^{(0)}$ has the Baire property. Then the following
properties are equivalent:
\begin{enumerate}
\item $G$ is essentially principal.
\item  The set of points of
$G^{(0)}$ with trivial isotropy is dense.
\end{enumerate}
\end{prop}

\begin{proof} Let us introduce the set $Y$ of units  with trivial
isotropy and its complement $Z=G^{(0)}\setminus Y$. Condition $(ii)$ means that $Z$ has
an empty interior.
\smallskip 
Let us show that
$(ii)\Rightarrow (i)$. Let
$U$ be an open subset of $G$ contained in $G'$. Since $G$ is Hausdorff,  $ G^{(0)}$
is closed in $G$ and 
$U\setminus G^{(0)}$ is open.
Therefore
$r(U\setminus G^{(0)})$, which is open  and contained in
$Z$, is empty. This implies that $U\setminus G^{(0)}$ itself is empty and that $U$
is contained in
$G^{(0)}$.
\smallskip
Let us show that
$(i)\Rightarrow (ii)$. We choose a countable family $(S_n)$ of open bisections which
cover $G$. We introduce the subsets $A_n=r(S_n\cap G')$ of $G^{(0)}$. By definition,
for each $n$, $Y_n=int(A_n)\cup ext(A_n)$ is a dense open subset of
$G^{(0)}$. By the Baire property, the intersection $\cap_nY_n$ is dense in
$G^{(0)}$. Let us show that
$\cap_nY_n$ is contained in $Y$. Suppose that $x$ belongs to $\cap_nY_n$ and that
$\gamma$ belongs to $G(x)$. There exists $n$ such that $\gamma$ belongs to $S_n$. Then
$\gamma$ belongs to $S_n\cap G'$ and $x=r(\gamma)$ belongs to $A_n$. Since it also
belongs to
$Y_n$, it must belong to $int(A_n)$. Let $V$ be an open set containing $x$ and contained
in $A_n$. Since $r$ is a bijection from $S_n\cap G'$ onto $A_n$, the open set $VS_n$
is contained in $G'$. According to $(i)$, it is contained in $G^{(0)}$ and 
$\gamma=xS_n$ belongs to $G^{(0)}$.
\end{proof}
\vskip3mm

\begin{prop} Let $\mathcal G$ be a pseudogroup on $X$,
let $G$ be its groupoid of germs and let $\mathcal S$ be the inverse semigroup of open
bisections of $G$. Then
\begin{enumerate} 
\item The pseudogroup
$\alpha({\mathcal S})$ is the ample pseudogroup $[{\mathcal G}]$ of $\mathcal G$.
\item The canonical action $\alpha$ is an isomorphism from $\mathcal S$ onto
$[{\mathcal G}]$.
\end{enumerate}
\end{prop}
\begin{proof} We have observed above that $\mathcal G$ and $[\mathcal G]$ define the
same groupoid of germs $G$. Thus, every $\varphi\in [\mathcal G]$ defines the open
bisection $S=S_\varphi={\mathcal U}(X,\varphi,X)$. By construction, $\alpha_S=\varphi$.
Conversely, let
$S$ be an open bisection of $G$. It can be written
as a union $S=\cup_i{\mathcal U}(V_i,\varphi_i,U_i)$, where $U_i,V_i$ are open subsets
of $X$ and $\varphi_i\in \mathcal G$. This shows that $\varphi=\alpha_S$ belongs to
$[\mathcal G]$ and that $S_\varphi=S$. In other words, the
maps $S\to\alpha_S$ and $\varphi\to S_\varphi$ are inverse of each other. 
\end{proof}

\begin{prop} Let $G$ be an \'etale groupoid over $X$ and let ${\mathcal S}$ be
the inverse semigroup of its open bisections. Let $\alpha$ be the canonical action of
$\mathcal S$ on $X$. The following properties are
equivalent:
\begin{enumerate}
\item The map $\alpha$ is one-to-one.
\item $G$ is essentially principal.
\end{enumerate}
\end{prop}
\begin{proof} Let $S$ be an open bisection of $G$. Clearly, $\alpha_S$ is an identity
map iff $S$ is contained in
$G'$. If $G$ is essentially principal, this implies that $S$ is contained in $G^{(0)}$.
Therefore, if $S$ and $T$ have the same image in ${\mathcal P}art(X)$, then $ST^{-1}\subset
G^{(0)}$ and $S=T$. Conversely, if $\gamma$ belongs to the interior of $G'$, it belongs
to some open bisection $S$ contained in $G'$.
Then
$\alpha_S$ is an identity map and by $(i)$, $S$ is contained in $G^{(0)}$. Therefore
$\gamma$ belongs to
$G^{(0)}$.
\end{proof}

A groupoid which satisfies condition $(i)$ of the above proposition is called {\it effective}. The reader will find a good discussion of this notion in the monograph \cite[pp 136-7]{mm:intro} by I.~Moerdijk and J.~Mr\~cun . The following characterization of groupoids of germs is well-known.

\begin{cor}\label{essentially principal}  Let  $G$ be an \'etale groupoid. Let $\mathcal
S$ be the inverse semigroup of its open bisections. The following properties are
equivalent:
\begin{enumerate}
\item $G$ is isomorphic to  the groupoid of germs of some pseudogroup.
\item $G$ is isomorphic to  the groupoid of germs of $\alpha({\mathcal S})$.
\item $G$ is essentially principal.
\end{enumerate}
\end{cor}

\begin{proof}
Assume that
$G$ is the groupoid of germs of a pseudogroup
$\mathcal G$ on $X$. We have seen that $\alpha({\mathcal S})$ is the ample pseudogroup
$[{\mathcal G}]$ of $\mathcal G$ and that the map
$S\in {\mathcal S}\mapsto \alpha_S\in{\mathcal G}_{\mathcal S}$ is one-to-one.
Conversely, assume that the \'etale groupoid $G$ on $X$ is essentially principal.
The
canonical action $\alpha$ 
identifies
$\mathcal S$ with the pseudogroup ${\mathcal G}=\alpha(\mathcal S)$.
Then the map $[x,S,y]\to Sy$ from the groupoid of germs $H$ of ${\mathcal G}$ to $G$
is a topological groupoid isomorphism. The surjectivity is clear, since $G$ can be
covered by open bisections. Suppose that $[x_i,S_i,y_i], i=1,2$ give the same element
$\gamma=S_1y_1=S_2y_2$. Then $y_1=y_2=y$ and the open bisections $S_1,S_2$ and $S_1\cap
S_2$ define the same germ at $y$. Hence $[x_1,S_1,y_1]=[x_2,S_2,y_2]$.
\end{proof}

\vskip 3mm
\section{The analysis of the twisted groupoid C$^*$-algebra.}

Following \cite{kum:diagonals}, one defines a twisted groupoid as a central groupoid 
extension
$${\bf T}\times G^{(0)}\rightarrow\Sigma\rightarrow G,$$
where ${\bf T}$ is the circle group. Thus, $\Sigma$ is a groupoid containing
${\bf T}\times G^{(0)}$ as a subgroupoid. One says that $\Sigma$ is a twist over
$G$.  We assume that $\Sigma$ and $G$ are topological groupoids. In particular,
$\Sigma$ is a principal ${\bf T}$-space and $\Sigma/{\bf T}=G$. We form the
associated complex line bundle $L=({\bf C}\times\Sigma)/{\bf T}$, where $\bf T$ acts
by the diagonal action
$z(\lambda,\sigma)=(\lambda\overline z,z\sigma)$. The class of $(\lambda,\sigma)$ is
written $(\lambda,\sigma)$. The line bundle
$L$ is a Fell bundle over the groupoid
$G$, as defined in \cite{kum:fell} (see also \cite{fel:bundles}: it has the product $L_{\dot\sigma}\otimes
L_{\dot\tau}\rightarrow L_{\dot{\sigma\tau}}$, sending
$([\lambda,\sigma],[\mu,\tau])$ into
$[\lambda\mu,\sigma\tau]$  and the involution
$L_{\dot\sigma}\rightarrow L_{\dot{\sigma}^{-1}}$ sending $[\lambda,\sigma]$ into
$[\overline\lambda,\sigma^{-1}]$. An element $u$ of a Fell bundle $L$ is
called unitary if $u^*u$ and $uu^*$ are unit elements. The set of unitary elements
of $L$ can be identified to
$\Sigma$ through the map $\sigma\in\Sigma\mapsto [1,\sigma]\in L$. In fact, this
gives a one-to-one correspondence between twists over $G$ and Fell line bundles over
$G$. (see \cite{dkr:fell}). It is convenient to view the sections of 
$L$ as complex-valued functions
$f:\Sigma\rightarrow {\bf C}$ satisfying $f(z\sigma)=f(\sigma)\overline z$ for all
$z\in{\bf T},\sigma\in\Sigma$ and we shall usually do so. When there is no risk of
confusion, we shall use the same symbol for the function $f$ and the section of $L$
it defines.

In order to define the twisted convolution algebra, we assume from now on that 
$G$ is locally compact, Hausdorff, second countable and that it possesses a Haar
system
$\lambda$. It is a family of
measures
$\{\lambda_x\}$ on $G$, indexed by $x\in G^{(0)}$, such that $\lambda_x$ has
exactly $G^x$ as its support, which is continuous, in the sense that for every
$f\in C_c(G)$, the function $\lambda(f):x\mapsto\lambda_x(f)$ is
continuous, and invariant, in the sense that for every $\gamma\in G$,
$R(\gamma)\lambda_{r(\gamma)}=\lambda_{s(\gamma)}$, where
$R(\gamma)\gamma'=\gamma'\gamma$. When $G$ is an \'etale groupoid, it has a
canonical Haar system, namely the counting measures on the fibers of $s$.

\medskip
Let $(G,\lambda)$ be a Hausdorff locally compact second countable groupoid with Haar
system and let
$\Sigma$ be a twist over $G$. We denote by $C_c(G,\Sigma)$ the space of continuous
sections with compact support of the line bundle associated with $\Sigma$. The
following operations
$$f*g(\sigma)=\int
f(\sigma\tau^{-1})g(\tau)d\lambda_{s(\sigma)}(\dot\tau)\quad\hbox{and}\quad
f^*(\sigma)=\overline{f(\sigma^{-1})}$$
turn $C_c(G,\Sigma)$ into a $*$-algebra.
Furthermore, we define for $x\in G^{(0)}$ the Hilbert space
$H_x=L^2(G_x,L_x,\lambda_x)$ of square-integrable sections
of the line bundle $L_x=L_{|G_x}$.
Then, for $f\in C_c(G,\Sigma)$, the operator $\pi_x(f)$ on $H_x$ defined by
$$\pi_x(f)\xi(\sigma)=\int f(\sigma\tau^{-1})\xi(\tau)d\lambda_x(\dot\tau)$$
is bounded. This can be deduced from the useful estimate:
$$\|\pi_x(f)\|\le\|f\|_I=\max{(\sup_y\int |f|d\lambda_y,\sup_y\int
|f^*|d\lambda_y)}.$$ Moreover, the field
$x\mapsto \pi_x(f)$ is continuous when the
family of Hilbert spaces
$H_x$ is given the structure of a continuous field of Hilbert spaces by choosing
$C_c(G,\Sigma)$ as a fundamental family of continuous sections. Equivalently, the
space of sections $C_0(G^{(0)},H)$ is a right C$^*$-module over $C_0(G^{(0)})$ and
$\pi$ is a representation of $C_c(G,\Sigma)$ on this C$^*$-module.
\smallskip
The reduced C$^*$-algebra
$C^*_{red}(G,\Sigma)$ is the completion of $C_c(G,\Sigma)$ with respect to the norm
$\|f\|=\sup_x\|\pi_x(f)\|$.
\bigskip

Let us now study the properties of the pair $(A=C^*_{red}(G,\Sigma),
B=C_0(G^{(0)}))$ that we have constructed from a twisted \'etale Hausdorff locally
compact second countable groupoid $(G,\Sigma)$.


The main technical tool is that the elements of the reduced C$^*$-algebra
$C^*_{red}(G,\Sigma)$ are still functions on $\Sigma$ (or sections of the line bundle
$L$).

\begin{prop}(\cite[2.4.1]{ren:approach}) Let $G$ be an \'etale Hausdorff locally
compact second countable groupoid and let $\Sigma$ be a twist over $G$. Then, for all
$f\in C_c(G,\Sigma)$ we have :
\begin{enumerate}
\item $|f(\sigma)|\le \|f\|$ for every $\sigma\in\Sigma$ and
\item $\int |f|^2 d\lambda_x\le \|f\|^2$ for every $x\in G^{(0)}$.
\end{enumerate}
\end{prop} 
\smallskip
{\it Proof.} This is easily deduced (see \cite[2.4.1]{ren:approach}) from the
following equalities:
$$f(\sigma)=<\epsilon_\sigma,\pi_{s(\sigma)}(f)\epsilon_{s(\sigma)}>,\quad
f_{|\Sigma_x}=\pi_x(f)\epsilon_x,$$
where
$f\in C_c(G,\Sigma)$, $\sigma\in\Sigma$, $x\in
G^{(0)}$ and $\epsilon_\sigma\in H_{s(\sigma)}$ is defined by
$\epsilon_\sigma(\tau)=\overline z$ if $\tau=z\sigma$ and $0$ otherwise.

As a consequence (\cite[2.4.2]{ren:approach}) the elements of
$C^*_{red}(G,\Sigma)$ can be viewed as continuous sections of the line bundle $L$.
Moreover, the above expressions of
$f*g(\sigma)$ and
$f^*(\sigma)$ are still valid for $f,g\in C^*_{red}(G,\Sigma)$ (the sum defining
$f*g(\sigma)$ is convergent). 
It will be convenient to define the open support of a
continuous section $f$ of the line bundle $L$ as 
$$supp'(f)=\{\gamma\in G: f(\gamma)\not=0\}.$$
Note that the unit space $G^{(0)}$ of $G$
is an open (and closed) subset of $G$ and that the restrictions of the twist $\Sigma$
and of the line bundle
$L$ to
$G^{(0)}$ are trivial. We have the following identification:
$$C_0(G^{(0)})=\{f\in C^*_{red}(G,\Sigma): supp'(f)\subset G^{(0)}\}$$
where $h\in C_0(G^{(0)})$ defines the section $f$ defined by $f(\sigma)=h(x)\overline z$
if $\sigma=(x,z)$ belongs to $G^{(0)}\times{\bf T}$ and $f(\sigma)=0$ otherwise.
Then $B=C_0(G^{(0)})$ is an abelian sub-C$^*$-algebra of $A=C^*_{red}(G,\Sigma)$ which
contains an approximate unit of $A$.

Here is an important application of the fact that the elements of
$C^*_{red}(G,\Sigma)$ can be viewed as continuous sections.

\begin{prop}(\cite[2.4.7]{ren:approach}) Let $(G,\Sigma)$ be a twisted \'etale
Hausdorff locally compact second countable groupoid.
Let
$A=C^*_{red}(G,\Sigma)$ and $B=C_0(G^{(0)})$. Then 
\begin{enumerate}
\item an element $a\in A$ commutes with every element of $B$
iff its open support $supp'(a)$ is contained in $G'$;
\item $B$ is a {\it masa} iff $G$ is essentially principal.
\end{enumerate}
\end{prop}
\begin{proof} Since the elements of $C^*_{red}(G,\Sigma)$ are continuous sections of
the associated line bundle $L$, it is straightforward to spell out the condition
$hf=fh$ for all $h\in B$. We refer to \cite[2.4.7]{ren:approach} for details.
\end{proof}


Another piece of structure of the pair $(A=C^*_{red}(G,\Sigma),
B=C_0(G^{(0)}))$ is the restriction map $P:f\mapsto
f_{|G^{(0)}}$ from $A$ to $B$. 

\begin{prop}\label{conditional expectation}(cf.\cite[2.4.8]{ren:approach})  Let $(G,\Sigma)$ be a twisted \'etale
Hausdorff locally compact second countable groupoid. Let
$P:C^*_{red}(G,\Sigma)\rightarrow C_0(G^{(0)})$ be the restriction map. Then
\begin{enumerate}
\item $P$ is a conditional expectation onto $C_0(G^{(0)})$.
\item $P$ is faithful.
\item If $G$ is essentially principal, $P$ is the unique conditional expectation
onto $C_0(G^{(0)})$.
\end{enumerate}
\end{prop}
\begin{proof} This is proved in \cite[2.4.8]{ren:approach} in the principal case.
The main point of
$(i)$ is that $P$ is well defined, which is clear from the above. There is no
difficulty checking that it has all the properties of an expectation map. Note that
for $h\in C_0(G^{(0)})$ and $f\in C^*_{red}(G,\Sigma)$, we have
$(hf)(\sigma)=h(r(\sigma))f(\sigma)$ and
$(fh)(\sigma)=f(\sigma)h(s(\sigma))$. The assertion $(ii)$ is also clear: for $f\in
C^*_{red}(G,\Sigma)$ and $x\in G^{(0)}$, we have $$P(f^* * f)(x)=\int |f(\tau)|^2
d\lambda_x(\dot\tau).$$ Hence, if $P(f^* * f)=0$, $f(\tau)=0$ for all
$\tau\in\Sigma$. Let us prove $(iii)$. Let
$Q:C^*_{red}(G,\Sigma)\rightarrow C_0(G^{(0)})$ be a conditional expectation. We shall
show that $Q$ and $P$ agree on $C_c(G,\Sigma)$, which suffices to prove the assertion.
Let $f\in C_c(G,\Sigma)$ with compact support $K$ in $G$. We first consider the case
when $K$ is contained in an open bisection $S$ which does not meet $G^{(0)}$ and show
that $Q(f)=0$. If $x\in G^{(0)}$ does not belong to $s(K)$, then $Q(f)(x)=0$. Indeed,
we choose $h\in C_c(G^{(0)})$ such that $h(x)=1$ and its support does not meet $s(K)$.
Then $fh=0$, therefore $Q(f)(x)=Q(f)(x)h(x)=(Q(f)h)(x)=Q(fh)(x)=0$. Let $x_0\in
G^{(0)}$ be such that $Q(f)(x_0)\not=0$. Then $Q(f)(x)\not=0$ on an open neighborhood
$U$ of $x_0$. Necessarily, $U$ contained in
$s(S)$. Since $G$ is essentially principal and $S$ does not meet $G^{(0)}$, the induced
homeomorphism $\alpha_S: s(S)\rightarrow r(S)$ is not the identity map on
$U$. Therefore, there exists $x_1\in U$ such that
$x_2=\alpha_S(x_1)\not=x_1$. We choose $h\in C_c(G^{(0)})$ such that $h(x_1)=1$ and
$h(x_2)=0$. We have
$hf=f (h\circ\alpha_S)$. Therefore,
$$
Q(f)(x_1)=h(x_1)Q(f)(x_1)=Q(hf)(x_1)=Q(f (h\circ\alpha_S))(x_1)=Q(f)(x_1)h(x_2)=0.$$
This is a contradiction. Therefore $Q(f)=0$.
Next, let us consider an arbitrary $f\in C_c(G,\Sigma)$ with compact support
$K$ in $G$. We use the fact that
$G^{(0)}$ is both open and closed in $G$. The compact set
$K\setminus G^{(0)}$ can be covered by finitely many open bisections $S_1,\ldots, S_n$ of
$G$. Replacing if necessary $S_i$ by $S_i\setminus G^{(0)}$, we may assume that $S_i\cap
G^{(0)}=\emptyset$. We set $S_0=G^{(0)}$. We introduce a partition of unity
$(h_0,h_1,\ldots,h_n)$ subordinate to the open cover $(S_0,S_1,\ldots, S_n)$ of
$K$: for all $i=0,\ldots,n$, $h_i:G\rightarrow [0,1]$ is continuous, it has a compact
support contained in $S_i$ and $\sum_{i=0}^n h_i(\gamma)=1$ for all $\gamma\in K$. We
define $f_i\in C_c(G,\Sigma)$ by $f_i(\sigma)=h_i(\dot\sigma)f(\sigma)$. Then, we have
$f=\sum_{i=0}^n f_i$, $f_0=P(f)$ and $f_i$ has its support contained in $S_i$ for all
$i$. Since $f_0\in C_0(G^{(0)})$, $Q(f_0)=f_0$. On the other hand, according to the
above, $Q(f_i)=0$ for $i=1,\ldots,n$. Therefore, $Q(f)=f_0=P(f)$.
\end{proof}

The C$^*$-module $C_0(G^{(0)},H)$ over $C_0(G^{(0)})$ introduced earlier to define
to define the representation $\pi$ and the reduced norm on
$C_c(G,\Sigma)$ is the completion of $A$ with respect to the
$B$-valued inner product $P(a^*a')$; the representation $\pi$ is left multiplication.

The conditional expectation $P$ will be used to recover the elements of $A$ as sections
of the line bundle $L$:

\begin{lem}\label{formula}  Let $(G,\Sigma)$ be a twisted \'etale
Hausdorff locally compact second countable groupoid. Let
$P: A=C^*_{red}(G,\Sigma)\rightarrow B=C_0(G^{(0)})$ be the restriction map. Then we
have the following formula: for all $\sigma\in \Sigma$, for all $n\in A$ such that
$supp'(n)$ is a bisection containing $\dot\sigma$ and all $a\in A$:
$$P(n^*a)(s(\sigma))=\overline{n(\sigma)}a(\sigma).$$

\end{lem}
\begin{proof} This results from the definitions. 
\end{proof}

The last property of the subalgebra $B=C_0(G^{(0)}))$ of $(A=C^*_{red}(G,\Sigma)$
which interests us is that it is regular. This requires the notion of normalizer as
introduced by A.~Kumjian.

\begin{defn} (\cite[1.1]{kum:diagonals}) Let $B$ be a sub C$^*$-algebra of
a C$^*$-algebra $A$. 
\begin{enumerate}
\item Its {\it normalizer} is the set
$$N(B)=\{n\in A: aBa^*\subset B\quad\hbox{and}\quad n^*Bn\subset B\}.$$
\item One says that $B$ is {\it regular} if its normalizer $N(B)$ generates $A$ as a
C$^*$-algebra.
\end{enumerate}
\end{defn}

Before studying the normalizer of $C_0(G^{(0)}))$ in $C^*_{red}(G,\Sigma)$, let us
give some consequences of this definition. We first observe that $B\subset N(B)$ and
$N(B)$ is closed under multiplication and involution. It is also a closed subset of
$A$. We shall always assume that $B$ contains an approximate unit of $A$. This condition is automatically satisfied when $B$ is maximal abelian and $A$ has a unit but this is not so in general (see \cite{was:masas}). We then
have the following obvious fact.

\begin{lem} Assume that $B$ be is a sub C$^*$-algebra of a
C$^*$-algebra
$A$ containing an approximate unit of $A$. Let $n\in N(B)$.
Then $nn^*, n^*n\in B$.
\end{lem}

Assume also that $B$ is abelian.
Let $X=\hat B$ so that $B=C_0(X)$. For $n\in N(B)$, define $dom(n)=\{x\in X:
n^*n(x)>0\}$ and
$ran(n)=\{x\in X: nn^*(x)>0\}$. These are open subsets of $X$.

\begin{prop}(\cite[1.6]{kum:diagonals}) Given $n\in N(B)$, there exists a unique
homeomorphism
$\alpha_n: dom(n)\rightarrow ran(n)$ such that, for all $b\in B$ and all
$x\in dom(n)$,
$$n^*bn(x)=b(\alpha_n(x)) n^*n(x).$$
\end{prop}
\begin{proof}
See \cite{kum:diagonals}.
\end{proof}
\bigskip

\begin{prop} Let $(G,\Sigma)$ be a twisted \'etale
Hausdorff locally compact second countable groupoid. Let
$A=C^*_{red}(G,\Sigma)$ and $B=C_0(G^{(0)})$ be as above. Then 
\begin{enumerate}
\item If the open support $S=supp'(a)$ of  $a\in A$ is a
bisection of $G$, then $a$ belongs to $N(B)$ and $\alpha_a=\alpha_S$;
\item If $G$ is essentially principal, the converse is true. Namely the normalizer
$N(B)$ consists exactly of the elements of
$A$ whose open support is a bisection.
\end{enumerate} 
\end{prop}
\begin{proof}
Suppose that $S=supp'(a)$ is a
bisection. Then, for $b\in B$,
$$a^*ba(\sigma)=\int
\overline{a(\tau\sigma^{-1})}b\circ r(\tau)a(\tau)d\lambda_{s(\sigma)}(\dot\tau).$$
The integrand is zero unless $\dot\tau\in S$ and $\dot{\tau\sigma^{-1}}\in S$, which
implies that
$\dot\sigma$ is a unit. Therefore $supp'(a^*ba)\subset G^{(0)}$ and $a^*ba\in B$.
Similarly, $aba^*\in B$. Moreover, if $\dot\sigma=x$ is a unit, we must have
$\dot\tau=Sx$ and therefore 
$$a^*ba(x)=a^*a(x)b\circ r(Sx)=a^*a(x)b\circ \alpha_S(x).$$
This shows that $\alpha_a=\alpha_S$.

Conversely, let us assume that $a$ belongs to $N(B)$. Let $S=supp'(a)$. Let us fix $x\in
dom(a)$. The equality
$$b(\alpha_a(x))=\int
{|a(\tau)|^2\over a^*a(x)}b\circ r(\tau)d\lambda_x(\dot\tau)$$
holds for all $b\in B$. In other words, the pure state $\delta_{\alpha_a(x)}$ is
expressed as a (possibly infinite) convex combination of pure states. This implies that
$a(\tau)=0$ if $r(\tau)\not=\alpha_a(x)$. Let
$$T=\{\gamma\in G: s(\gamma)\in dom(a),\hbox{and}\, r(\gamma)=\alpha_a\circ
s(\gamma)\}.$$
We have established the containment $S\subset T$. This implies $SS^{-1}\subset
TT^{-1}\subset G'$. If $G$ is essentially principal, $SS^{-1}$ which is open must be
contained in $G^{(0)}$. Similarly, $S^{-1}S$ must be
contained in $G^{(0)}$. This shows that $S$ is a bisection.
\end{proof}

\begin{cor} Let $(G,\Sigma)$ be a twisted \'etale
Hausdorff locally compact second countable groupoid. Let
$A=C^*_{red}(G,\Sigma)$. Then 
$B=C_0(G^{(0)})$ is a regular sub-C$^*$-algebra of $A$. 
\end{cor}
\begin{proof} Since $G$ is \'etale, the open bisections of $G$ form a basis of open
sets for $G$. Every element
$f\in C_c(G,\Sigma)$ can be written as a finite sum of sections supported by open
bisections. Thus the linear span of $N(B)$ contains $C_c(G,\Sigma)$. Therefore,
$N(B)$ generates $A$ as a C$^*$-algebra.
\end{proof}

Let us continue to investigate the properties of the normalizer $N(B)$

\begin{lem}(\cite[1.7]{kum:diagonals}) Let $B$ be a
sub-C$^*$-algebra of a C$^*$-algebra $A$. Assume that $B$ is abelian and contains an
approximate unit of $A$. Then
\begin{enumerate}
\item If $b\in B$, $\alpha_b=id_{dom(b)}$.
\item If $m,n\in N(B)$, $\alpha_{mn}=\alpha_m\circ\alpha_n$ and
$\alpha_{n^*}=\alpha_n^{-1}$.
\end{enumerate} 
\end{lem}

This shows that ${\mathcal G}(B)=\{\alpha_a, a\in N(B)\}$ is a pseudogroup on $X$.
By analogy with the canonical action of the inverse semigroup of open bisections of
an \'etale groupoid, we call the map $\underline\alpha: N(B)\rightarrow {\mathcal
G}(B)$ such that $\underline\alpha(n)=\alpha_n$ the canonical action of the
normalizer.

\begin{defn}
We shall say that ${\mathcal G}(B)$ is the {\it Weyl pseudogroup} of $(A,B)$. We define
the Weyl groupoid of $(A,B)$ as the groupoid of germs  of ${\mathcal G}(B)$.
\end{defn}

\begin{prop}\label{commutant} Let $B$ be a
sub-C$^*$-algebra of a C$^*$-algebra $A$. Assume that $B$ is abelian and contains an
approximate unit of $A$. Then:
\begin{enumerate}
\item The kernel of the canonical action $\underline\alpha: N(B)\rightarrow {\mathcal
G}(B)$ is the commutant $N(B)\cap B'$ of $B$ in $N(B)$. 
\item If $B$ is maximal abelian, then $ker\underline\alpha=B$.
\item If $B$ is regular and $ker\underline\alpha=B$, then $B$ is maximal abelian.
\end{enumerate}
\end{prop}

\begin{proof} If $n\in N(B)\cap B'$, then for all $b\in B$, $n^*bn=bn^*n$. By
comparing with the definition of $\alpha_n$, we see that $\alpha_n(x)=x$ for all
$x\in dom(n)$. Conversely, suppose that $n\in N(B)$ satisfies $n^*bn(x)=b(x)
n^*n(x)$ for all $b\in B$ and all $x\in dom(n)$. We also have $n^*bn(x)=b(x)
n^*n(x)=0$ when $x\notin dom(n)$ because of the inequality $0\le n^*bn\le\|b\|n^*n$
for $b\in B_+$. Therefore
$n^*bn=bn^*n$ for all
$b\in B$. As observed in \cite[1.9]{kum:diagonals}, this implies that
$(nb-bn)^*(nb-bn)=0$ for all $b\in B$ and $nb=bn$ for all $b\in B$. The
assertions $(ii)$ and
$(iii)$ are immediate consequences of $(i)$.
\end{proof}

Let us study  the normalizer $N(B)$ in our particular situation, where
$A=C^*_{red}(G,\Sigma)$ and $B=C_0(G^{(0)})$. 
\bigskip

\begin{prop} Let $(G,\Sigma)$ be a twisted \'etale
Hausdorff locally compact second countable groupoid. Let
$A=C^*_{red}(G,\Sigma)$ and $B=C_0(G^{(0)})$ be as above. Assume that $G$ is
essentially principal. Then,
\begin{enumerate}
\item the Weyl pseudogroup ${\mathcal G}(B)$ of $(A,B)$ consists of the partial
homeomorphisms $\alpha_S$ where $S$ is an open bisection of $G$ such that
the restriction of the associated line bundle $L$ to
$S$ is trivializable;
\item the Weyl groupoid $G(B)$ of
$(A,B)$ is canonically isomorphic to $G$.
\end{enumerate} 
\end{prop}
\begin{proof} Recall that ${\mathcal S}$ denotes the inverse semigroup of open
bisections of $G$ and ${\mathcal G}$ denotes the pseudogroup defined by ${\mathcal
S}$. We have defined the canonical action
$\alpha: {\mathcal S}\rightarrow {\mathcal G}$ and the canonical action
$\underline\alpha: N(B)\rightarrow {\mathcal G}(B)$. We have seen that $\alpha$
and $\underline\alpha$ are related by $\underline\alpha=\alpha\circ supp'$, where
$supp'(n)$ denotes the open support of
$n\in N(B)$,
Moreover, the restriction of the line bundle to $S=supp'(n)$ is
trivializable, since it possesses a non-vanishing section. Conversely, let
$S$ be an open bisection such that the restriction $L_{|S}$ is trivializable. Let us
choose a non-vanishing continuous section $u:S\rightarrow L$. Replacing $u(\gamma)$ by
$u(\gamma)/\|u(\gamma)\|$, we may assume that $\|u(\gamma)\|=1$ for all $\gamma\in S$.
Then, we choose 
$h\in C_0(G^{(0)})$ such that $supp'(h)=s(S)$ and define the section $n:G\rightarrow
L$ by $n(\gamma)=u(\gamma)h\circ s(\gamma)$ if $\gamma\in S$ and $n(\gamma)=0$
otherwise. Let $(h_i)$ be a sequence in $C_c(G^{(0)})$, with $supp(h_i)\subset s(S)$,
converging uniformly to $h$. Then $uh_i\in C_c(G,\Sigma)$ and the sequence $(uh_i)$
converges to
$n$ in the norm $\|.\|_I$ introduced earlier. This implies that $n$ belongs to $A$.
We have $S=supp'(n)$ as desired. This shows that ${\mathcal G}(B)$ is exactly the
pseudogroup consisting of the partial homeomorphisms $\alpha_S$ such that
$S$ is an open bisection of $G$ on which $L$ is trivializable. According to a theorem
of Douady and Soglio-H\'erault (see Appendix of \cite{fel:bundles}), for all open
bisection
$S$ and all
$\gamma\in S$, there exists an open neighborhood $T$ of $\gamma$ contained in $S$ on
which $L$ is trivializable. Therefore ${\mathcal G}(B)$ and the pseudogroup $\mathcal
G$ defined by all open bisections have the same groupoid of germs, which is
isomorphic to $G$ by 
\corref{essentially principal}.
\end{proof}

Let us see next how the twist $\Sigma$ over $G$ can be recovered from the pair $(A,B)$.
This is done exactly as in Section 3 of \cite{kum:diagonals}. Given an abstract pair
$(A,B)$, we set $X=\hat B$ and introduce
$$D=\{(x,n,y)\in X\times N(B)\times X: n^*n(y)>0\,\,\hbox{and}\,\, x=\alpha_n(y)\}$$
and its quotient $\Sigma(B)=D/\sim$ by the equivalence relation:
$(x,n,y)\sim (x', n',y')$ if and only if $y=y'$ and there exist $b,b'\in B$ with
$b(y),b'(y)>0$ such that $nb=n'b'$. The class of $(x,n,y)$ is denoted by $[x,n,y]$.
Now
$\Sigma(B)$ has a natural structure of groupoid over $X$, defined exactly in the same
fashion as a groupoid of germs:  the range and source maps are defined by
$r[x,n,y]=x, s[x,n,y]=y$, the product by $[x,n,y][y,n',z]=[x,nn',z]$ and the
inverse by $[x,n,y]^{-1}=[y,n^*,x]$.

The map $(x,n,y)\to [x,\alpha_n,y]$ from $D$ to $G(B)$ factors through the
quotient and defines a groupoid homomorphism from $\Sigma(B)$ onto $G(B)$. Moreover
the subset ${\mathcal B}=\{[x,b,x]: b\in B, b(x)\not=0\}\subset\Sigma(B)$ can be
identified with the trivial group bundle ${\bf T}\times X$ via the map
$[x,b,x]\mapsto (b(x)/|b(x)|,x)$. In general, ${\mathcal
B}\rightarrow\Sigma(B)\rightarrow G(B)$ is not an extension, but this is the case
when $B$ is maximal abelian.

\begin{prop}\label{twist} Assume that $B$ is a masa in $A$ containing an
approximate unit of $A$. Then
$${\mathcal
B}\rightarrow\Sigma(B)\rightarrow G(B)$$
is (algebraically) an extension.
\end{prop}
\begin{proof} We have to check that an element $[x,n,y]$ of $\Sigma(B)$ which has a
trivial image in $G(B)$ belongs to $\mathcal B$. If the germ of $\alpha_n$ at $y$ is
the identity, then $x=y$ and we have a neighborhood $U$ of $y$ contained in $dom(n)$
such that
$\alpha_n(z)=\alpha_{n^*}(z)=z$ for all $z\in U$.  
We choose
$b\in B$ with compact support contained in $U$ and such that
$b(x)>0$  and we define
$n'=nb$. Then $\alpha_{n'}$ is trivial. According to
\propref{commutant}, $n'$
belongs to $B$ and $[x,n,x]=[x,n',x]$ belongs to $\mathcal B$.
\end{proof}

We shall refer to $\Sigma(B)$ as the Weyl twist of the pair $(A,B)$.

\begin{prop} Let $(G,\Sigma)$ be a twisted \'etale
Hausdorff locally compact second countable essentially
principal groupoid. Let
$A=C^*_{red}(G,\Sigma)$ and $B=C_0(G^{(0)})$ be as above. Then we have a canonical
isomorphism of extensions:
$$\begin{CD}
 {\mathcal B} @>{ }>>\Sigma(B) @>{}>>G(B) \\
@V{}VV      @V{}VV      @V{}VV\\
{\bf T}\times G^{(0)} @>{ }>> \Sigma @>{ }>> G
\end{CD}$$ 
\end{prop}
\begin{proof} The left and right vertical arrows have been already defined and shown
to be isomorphisms. It suffices to define the middle vertical arrow and show that it
is a groupoid homomorphism which makes the diagram commutative. Let
$(x,n,y)\in D$. Since $n$ belongs to $N(B)$ and $G$ is essentially principal,
$S=supp'(n)$ is an open bisection of $G$. The element $n(Sy)/\sqrt{n^*n(y)}$ is a
unitary element of $L$ because $n^*n(y)=\|n(Sy)\|^2$ and can therefore be viewed as
an element of $\Sigma$. Let $(x,n',y)\sim(x,n,y)$. There exist $b,b'\in B$ with
$b(y),b'(y)>0$ such that $nb=n'b'$. This implies that the open supports $S=supp'(n)$
and
$S=supp'(n)$ agree on some neighborhood of $Sy$. In particular, $Sy=S'y$ . Moreover,
the equality $n(Sy)b(y)=n'(Sy)b'(y)$ implies that
$n(Sy)/\sqrt{n^*n(y)}=n'(Sy)/\sqrt{{n'}^*n'(y)}$. Thus we have a well-defined map
$\Phi:[x,n,y]\mapsto (n(Sy)/\sqrt{n^*n(y)}, Sy)$ from $\Sigma(B)$ to $\Sigma$. Let us
check that it is a groupoid homomorphism. Suppose that we are given
$(x,m,y),(y,n,z)\in D$. Let
$S=supp'(m)$,
$T=supp'(n)$. Then $supp'(mn)=ST$. We have to check the equality
$${mn(STz)\over
\sqrt{(mn)^*mn(z)}}={m(Sy)\over\sqrt{m^*m(y)}}{n(Tz)\over\sqrt{n^*n(z)}}.$$
It is satisfied because $mn(STz)=m(Sy)n(Tz)$ and 
$$(mn)^*(mn)(z)=(n^*(m^*m)n)(z)=(m^*m)(\alpha_n(z))n^*n(z)=m^*m(y)n^*n(z).$$
The image of $[x,n,y]^{-1}=[y,n^*,x]$ is
$n^*(S^{-1}x)/\sqrt{nn^*(x)}=(n(xS))^*/\sqrt{nn^*(x)}$. It is the inverse of 
$n(Sy)/\sqrt{n^*n(y)}$ because $xS=Sy$ and $nn^*(x)=n^*n(y)$ and the involution
agrees with the inverse on $\Sigma\subset L$. Let us check that we have a commutative
diagram. The restriction of $\Phi$ to $\mathcal B$ sends $[x,b,x]$ to
$(b(x)/|b(x)|,x)$. This is exactly the left vertical arrow. The image of $[x,n,y]$
in $G(B)$ is the germ $[x,\alpha_n,y]$. The image of $(n(Sy)/\sqrt{n^*n(y)}, Sy)$ in
$G$ is $Sy$. The map $[x,\alpha_n,y]\mapsto Sy$ is indeed the canonical isomorphism
from $G(B)$ onto $G$.

\end{proof}
In the previous proposition, we have viewed $\Sigma(B)$ as an algebraic extension.
It is easy to recover the topology of $\Sigma(B)$. Indeed, as we have already seen,
every $n\in N(B)$ defines a trivialization of the restriction of $\Sigma(B)$ to the
open bisection $S=supp'(n)$. This holds in the abstract framework.  Assume that $B$
is a masa in $A$. Let $n\in N(B)$. Its open support is by definition the open
bisection $S\subset G(B)$ which induces the same partial homeomorphism as $n$. We
define the bijection
$$\varphi_n: {\bf T}\times dom(n)\rightarrow \Sigma(B)_{|S},$$
by $\varphi_n(t,x)=[\alpha_n(x),tn,x]$.

\begin{lem} (cf.\cite[Section 3]{kum:diagonals})  Assume that $B$ is a masa in $A$ containing an approximate unit of $A$. With above notation, 
\begin{enumerate}
\item Two elements $n_1,n_2\in N(B)$ which have the same open support $S$ define
compatible trivializations of $\Sigma(B)_{|S}$. 
\item $\Sigma(B)$ is a locally trivial topological twist over $G(B)$.
\end{enumerate} 
\end{lem}

\begin{proof} For $(i)$, assume that $n_1$ and $n_2$ have the same open support $S$.
Then, according to
\propref{commutant}, there exist $b_1,b_2\in B$, non vanishing on $s(S)$ and such
that $n_1b_1=n_2b_2$. A simple computation from the relation
$\varphi_{n_1}(t_1,x)=\varphi_{n_2}(t_2,x)$ and the fact that for $n\in N(B)$ and
$b\in B$, the equality $nb=0$ implies $b(x)=0$ whenever $n^*n(x)>0$ gives $t_2=t_1
u(x)$ where $u(x)={b_2(x)|b_1(x)|\over |b_2(x)| b_1(x)}$. Therefore, the transition
function is a homeomorphism. We deduce $(ii)$. Indeed, we have given a topology to
$\Sigma(B)_{|S}$ whenever $S$ is a bisection arising from the Weyl pseudogroup
${\mathcal G}(B)$. This family, which is stable under finite intersection and which
covers
$\Sigma(B)$, is a base of open sets for the desired topology.
\end{proof}

\vskip 3mm
\section{Cartan subalgebras in C$^*$-algebras}

Motivated by the properties of the pair $(A=C^*_{red}(G,\Sigma), B=C_0(G^{(0)}))$
arising from a twisted \'etale locally compact second countable Hausdorff essentially
principal groupoid, we make the following definition, analogous to \cite[Definition
3.1]{fm:relations II} of a Cartan subalgebra in a von Neumann algebra. We shall
always assume that the ambient C$^*$-algebra $A$ is separable.

\begin{defn} We shall say that an abelian
sub-C$^*$-algebra $B$ of a C$^*$-algebra $A$
is a {\it Cartan subalgebra} if
\begin{enumerate}
\item $B$ contains an approximate unit of $A$;
\item $B$ is maximal abelian;
\item $B$ is regular;
\item there exists a faithful conditional expectation $P$ of $A$ onto $B$.
\end{enumerate}
Then $(A,B)$ is called a Cartan pair.
\end{defn}

Let us give some comments about the  definition. First, when $A$ has a unit, a maximal abelian sub-C$^*$-algebra
necessarily contains the unit; however, as said earlier,  there exist maximal abelian
sub-C$^*$-algebras which do not contain an approximate unit for the ambient
C$^*$-algebra. Since in our models, namely \'etale groupoid C$^*$-algebras, the subalgebra corresponding to the unit space always contains an approximate unit of $A$, we have to make this assumption. Second this definition of a Cartan subalgebra should be compared to the Definition 1.3 of a C$^*$-diagonal given by
A.~Kumjian in \cite{kum:diagonals} (see also \cite{ren:dual}): there it is assumed
that $B$ has
the unique extension property, a property introduced by J.~Anderson and studied by
R.~Archbold et alii. If $B$ has the unique extension property (and under the assumption that it contains an approximate unit of $A$), it is maximal abelian and there exists one and only one conditional expectation onto $B$. We shall say more about the unique extension property when we compare \thmref{Cartan} and Kumjian's theorem.

Given a Cartan pair $(A,B)$, we construct the normalizer $N(B)$, the Weyl groupoid
$G(B)$ on $X=\hat B$, the Weyl twist $\Sigma(B)$ and the associated line bundle
$L(B)$. In fact, these constructions can be made under the sole assumption that $B$
is a {\it masa}. Let us see how the elements of
$A$ define sections of the line bundle
$L(B)$ or equivalently, functions $f:\Sigma\rightarrow {\bf C}$ satisfying
$f(t\sigma)=\overline t f(\sigma)$ for all $t\in{\bf T}$ and $\sigma\in\Sigma(B)$.
The answer is given by \lemref{formula} (this formula also appears in
\cite{kum:diagonals}). Recall that $\Sigma(B)$ is defined as a quotient of
$$D=\{(x,n,y)\in X\times N(B)\times X: n^*n(y)>0\,\,\hbox{and}\,\, x=\alpha_n(y)\}.$$

\begin{lem}\label{evaluation} Given
$a\in A$ and $(x,n,y)\in D$, we define
$$\hat a(x,n,y)={P(n^*a)(y)\over\sqrt{n^*n(y)}}.$$
Then
\begin{enumerate}
\item $\hat a(x,n,y)$ depends only on its class in $\Sigma(B)$;
\item $\hat a$ defines a continuous section of the line bundle  $L(B)$;
\item the map $ a\mapsto\hat a$ is linear and injective.
\end{enumerate}
\end{lem}

\begin{proof}
Assertion $(i)$ is clear, since $\hat a(x,nb,y)=\hat a(x,n,y)$ for all $b\in B$ such
that $b(y)>0$. For $(ii)$, the equality $\hat a(x,tn,y)=\overline t\hat a(x,n,y)$
for all $t\in{\bf T}$ shows that $\hat a$ defines a section of $L(B)$. To get the
continuity, it suffices to check the continuity of $\hat a$ on the open subsets
$\Sigma(B)_{|S}$, where $S$ is the open support of $n\in N(B)$. But this is exactly
the continuity of the function $y\mapsto P(n^*a)(y)/\sqrt{n^*n(y)}$ on $dom(n)$. The
linearity in $(iii)$ is clear. Let us assume that $\hat a=0$. Let $n\in N(B)$. Then
$P(n^*a)(y)=0$ for all $y\in dom(n)$, hence also in its closure. If $y$ does not
belong to the closure of $dom(n)$, we can find $b\in B$ such that $b(y)=1$ and
$nb=0$. Then $P(n^*a)(y)=P(b^*n^*a)(y)=0$. Therefore $P(n^*a)=0$ for all $n\in
N(B)$. By regularity of $B$, this implies $P(a^*a)=0$. By faithfulness of $P$, this
implies that $a=0$.
\end{proof}

\begin{defn}\label{evaluation}  The map $\Psi:a\mapsto\hat a$ from $A$ to the space of continuous sections of $L(B)$ will be called the {\it evaluation map} of the Cartan pair $(A,B)$.
\end{defn}

\begin{lem}\label{separation}  Let $(A,B)$ be a Cartan pair. For $n\in N(B)$ and
$x\in dom(n)$ such that the germ of $\alpha_n$ at $x$ is not trivial, we have
$P(n)(x)=0$.
\end{lem}
\begin{proof} Since the germ of $\alpha_n$ at $x$ is not trivial, there exists a
sequence $(x_i)$ in $dom(n)$ which converges to $x$ and such that
$\alpha_n(x_i)\not=x_i$. We fix $i$. There exist  $b',b''\in B$ such that
$b'(x_i)=1$,
$b''(x_i)=0$ and
$b''n=nb'$. Indeed, there exists $b\in B$ with compact support contained in $ran(n)$
such that $b(\alpha_n(x_i))(n^*n)(x)=1$ and $b(x_i)=0$. Then
$b'=(b\circ\alpha_n)(n^*n)$ and $b''=(nn^*)b$ satisfy the conditions. We have
$$P(n)(x_i)=P(n)(x_i)b'(x_i)=P(nb')(x_i)=P(b''n)(x_i)=b''(x_i)P(n)(x_i)=0.$$
By continuity of $P(n)$, $P(n)(x)=0$. 
\end{proof}

\begin{cor}\label{support} Let $a\mapsto\hat a$ be the evaluation map of the Cartan pair $(A,B)$.
\begin{enumerate}
\item Suppose that $b$ belongs to $B$; then $\hat b$ vanishes off $X$ and its
restriction to
$X$ is its Gelfand transform.
\item Suppose that $n$ belongs to $N(B)$; then the open support of $\hat
n$ is the open bisection of $G(B)$ defined by the partial homeomorphism $\alpha_n$.
\end{enumerate}
\end{cor}

\begin{proof} Let us show $(i)$. If $\gamma=[\alpha_n(x),\alpha_n,x]\in G(B)$ is not
a unit, the germ of $\alpha_n$ at $x$ is not trivial. According to the lemma, for all
$b\in B$,
$P(n^*b)(x)=P(n)(x)b(x)=0$. Therefore, $\hat b(\gamma)=0$. On the other hand, if
$\gamma=x$ is a unit, $\hat b(x)=P({b_1}^*b)(x)={b_1}^*(x)b(x)=b(x)$ for $b_1\in B$
such that $b_(x)=1$. Let us show $(ii)$. If $n\in N(B)$, the lemma shows that $\hat
n[x,m,y]=0$ unless $y\in dom(n)$ and $\alpha_m$ has the same germ as $\alpha_n$ at
$y$. Then $[x,\alpha_m,y]=[x,\alpha_n,y]$ belongs to the open bisection $S_n$ of
$G(B)$ defined by the partial homeomorphism $\alpha_n$. On the other hand,
$\hat n(x,n,y)=n^*n(y)/\sqrt{n^*n(y)}$ is non zero for $y\in dom(n)$.
\end{proof}

\begin{prop} The Weyl groupoid $G(B)$ of a Cartan pair $(A,B)$ is a Hausdorff
\'etale groupoid. 
\end{prop}

\begin{proof} Let us show that the continuous functions $\hat a$, where $a\in A$
separate the points of $G(B)$ in the sense that for all $\sigma,\sigma'\in \Sigma$
such that $\dot\sigma\not=\dot\sigma'$, there exists $a\in A$ such that $\hat
a(\sigma)\not=0$ and $\hat a(\sigma')=0$. By construction, $\sigma=[x,n,y],
\sigma'=[x',n',y']$ where $n,n'\in N(B)$ and $y\in dom(n),y'\in dom(n')$. If
$y\not=y'$, we can take $a$ of the form $nb$ where $b(y)\not=0$ and
$b(y')=0$. If $y=y'$, since $\alpha_n$ and $\alpha_{n'}$ do not have the same germ
at $y$, we have by \lemref{separation} that $P({n'}^*n)(y)=0$, which implies $\hat
n(\sigma')=0$. On the other hand, $\hat n(\sigma)=\sqrt{n^*n(y)}$ is non-zero. We
can furthermore assume that $\hat
a(\sigma)=1$. Let $U=\{\tau: |\hat a(\tau)-1|<1/2\}$ and $V=\{\tau: |\hat
a(\tau)|<1/2\}$. Their images $\dot U,\dot V$ in $G(B)$ are open, disjoint and
$\dot\sigma\in\dot U,{\dot\sigma}'\in\dot V$.
\end{proof}

\begin{lem} Let $(A,B)$ be a Cartan pair. Let $N_c(B)$ be the set of elements $n$ in
$N(B)$ such $\hat n$ has compact support and let $A_c$ be its linear span. Then
\begin{enumerate}
\item $N_c(B)$ is dense in $N(B)$ and $A_c$ is dense in $A$;
\item the evaluation map $\Psi:a\mapsto\hat a$ defined above sends bijectively $A_c$ onto
$C_c(G(B),\Sigma(B))$ and $B_c=B\cap A_c$ onto $C_c(G^{(0)})$;
\item the evaluation map $\Psi:A_c\rightarrow C_c(G(B),\Sigma(B))$ is a $*$-algebra isomorphism. 
\end{enumerate}
\end{lem}

\begin{proof} For $(i)$, given $n\in N(B)$, there exists $b\in B$ such that $nb=n$.
There exists a sequence $(b_i)$ in $B$ such that $\hat b_i\in C_c(G^{(0)})$ and 
$(b_i)$ converges to $b$. Then $nb_i$ belongs to $N_c(B)$ and the sequence $(nb_i)$
converges to $n$. Note that $N_c(B)$ is closed under product and
involution and that $A_c$ is a dense
sub-$*$-algebra of $A$. Let us prove $(ii)$. By construction, $\Phi(A_c)$ is
contained in $C_c(G,\Sigma)$. The injectivity of $\Phi$ has been established in
\lemref{evaluation}. Let us show that $\Phi(A_c)=C_c(G,\Sigma)$. The family
of open bisections  $S_n=\{[\alpha_n(x),\alpha_n,x], x\in dom(n)\}$, where $n$ runs
over $N(B)$, forms an open cover of $G(B)$. If $f\in C_c(G,\Sigma)$ has its support
contained in $S_n$, then $\hat n$ is a non-vanishing continuous section over $S_n$
and there exists
$h\in C_c(G^{(0)})$ such that $f=\hat n h$. Since $h=\hat b$ with $b\in B_c$,
$f=\hat a$, where $a=nb$ belongs to $N_c(B)$. For a general $f\in C_c(G,\Sigma)$, we
use a partition of unity subordinate to a finite open cover $S_{n_1},\ldots,
S_{n_l}$ of the support of $f$. Let us prove $(iii)$. By linearity of $\Psi$, it suffices to check the relations $\Psi(mn)=\Psi(m)*\Psi(n)$ and $\Psi(n^*)=\Psi(n)^*$ for $m,n\in N(B)$. According to \corref{support}, $\Psi(mn)(\sigma)=0$ unless $\sigma=t[x,mn,z]$ with $z\in dom(mn)$ and $t\in{\bf T}$; then we  have
$$\Psi(mn)(t[x,mn,z])={\overline t}\sqrt {((mn)^*mn)(z)}.$$
On the other hand, $\Psi(m)\Psi(n)(\sigma)=0$ unless $\sigma$ is of the form 
$$\sigma=t[x,m,y][y,n,z]=t[x,mn,z]$$ and then
$$\Psi(m)\Psi(n)(t[x,mn,z])=\Psi(m)(t[x,m,y])\Psi(n)([y,n,z]={\overline t}\sqrt {(m^*m)(y)
(n^*n)(z)}.$$
The equality results from
$$(mn)^*(mn)(z)=(n^*(m^*m)n)(z)=(m^*m)(\alpha_n(z))n^*n(z)=m^*m(y)n^*n(z).$$
Similarly, $\Psi(n^*)(\sigma)=0$ unless $\sigma={\overline t}[y,n^*,x]$ with $x\in dom(n^*)$ and $t\in{\bf T}$ and then we  have
$$\Psi(n^*)({\overline t}[y,n^*,x])=t\sqrt {(nn^*)(x)}.$$
On the other hand, $\Psi(n)^*(\sigma)=\overline{\Psi(n)(\sigma^{-1})}=0$ unless $\sigma^{-1}=t[x,n,y]$ with $y\in dom(n)$ and $t\in{\bf T}$ and then we  have
$$\Psi(n)^*({\overline t}[y,n^*,x])=t\sqrt {(n^*n)(y)}.$$
These numbers are equal because $nn^*(x)=n^*n(y)$.

\end{proof}

\begin{thm}\label{Cartan} Let $(A,B)$ be a Cartan pair. Then the evaluation map $\Phi: a\mapsto\hat a$  is a C$^*$-algebra isomorphism of $A$ onto
$C_r^*(G(B),\Sigma(B))$ which carries $B$ onto $C_0(G^{(0)})$.
\end{thm}
\begin{proof} Let us show
that the evaluation map
$\Psi:A_c\rightarrow C_c(G,\Sigma)$ is an isometry with respect to the norms of $A$
and $C_r^*(G,\Sigma)$. Since
$P$ is faithful, we have for any $a\in A$ the equality
$$\|a\|=\sup\{\|P(c^*a^*ac)\|^{1/2}: c\in A_c, P(c^*c)\le 1\}.$$
If we assume that $a$ belongs to $A_c$, then $\hat a$ belongs to $C_r^*(G,\Sigma)$
and satisfies a similar formula:
$$\begin{array}{lll}
\|\hat a\|&=\sup\{\|\hat P(f^*(\hat a)^*\hat a f)\|^{1/2}: f\in C_c(G,\Sigma),
\hat P(f^*f)\le 1\}\\
&=\sup\{\|\hat P((\hat c)^*(\hat a)^*\hat a \hat c)\|^{1/2}: c\in A_c,
\hat P((\hat c)^*\hat c)\le 1\}.
\end{array}$$
Since $\Psi:A_c\rightarrow C_c(G,\Sigma)$ satisfies the
relation
$\hat P\circ\Psi=\Psi\circ P$, where $\hat P$ is the restriction map to $\hat B$, we
have the equality of the norms: $\|\hat c\|=\|c\|$. Hence
$\Psi$ extends to a C$^*$-algebra isomorphism $\tilde\Psi:A\rightarrow
C_r^*(G,\Sigma)$. By continuity of point evaluation, $\tilde\Psi(a)=\Psi(a)$ as
defined initially.
\end{proof}

We have mentioned earlier that the unique extension property of $B$ implies the uniqueness of the conditional expectation onto $B$. The uniqueness still holds for Cartan subalgebras.

\begin{cor} Let $B$ be a Cartan subalgebra of a C$^*$-algebra $A$. Then, there exists a unique expectation onto $B$.
\end{cor}

\begin{proof} This results from the above theorem and \propref{conditional expectation}.
\end{proof}

The following proposition is essentially a reformulation of Kumjian's theorem (see \cite{kum:diagonals} and \cite{ren:dual}). For the sake of completeness, we recall his proof. One says that the subalgebra $B$ has the unique extension property if every pure state of $B$ extends uniquely to a (pure) state of $A$. A C$^*$-diagonal is a Cartan subalgebra which has the unique extension property.

\begin{prop} (cf  \cite{kum:diagonals}, \cite{ren:dual}) Let $(A,B)$ be a Cartan pair. Then $B$ has the unique extension property if and only if the Weyl groupoid $G(B)$ is principal.
\end{prop}

\begin{proof} We may assume that $(A,B)=(C^*_r(G,\Sigma),C_0(G^{(0)}))$, where $G$ is an \'etale essentially principal Hausdorff  groupoid and $\Sigma$ is a twist over $G$. Suppose that $G$ is principal. A.~Kumjian shows that this implies that the linear span of the set $N_f(B)$ of free normalizers is dense in the kernel of the conditional expectation $P$, where a normalizer $n\in N(B)$ is said to be free if $n^2=0$. Indeed, since an arbitrary element of the kernel can be approximated by elements in  $C_c(G,\Sigma)\cap Ker(P)$, it suffices to consider a continuous section $f$ with compact support which vanishes on $G^{(0)}$. Since the compact support of $f$ does not meet the diagonal $G^{(0)}$, which is both open and closed, it admits a finite cover by open bisections $U_i$ such that $r(U_i)\cap s(U_i)=\emptyset$. Let $(h_i)$ be a partition of unity subordinate to the open cover $(U_i)$. Then, $f=\sum g_i$, where $g_i(\sigma)=f(\sigma)h_i(\dot\sigma)$ is a free normalizer. Then, he observes that free normalizers are limits of commutators $ab-ba$, with $a\in A$ and $b\in B$. This show that $A=B+\overline{span}[A,B]$, which is one of the characterizations of the extension property given in Corollary 2.7 of \cite{abg:extensions}. We suppose now that $B$ has the unique extension property and we show that the isotropy of $G$ is reduced to $G^{(0)}$. It suffices to show that for $n\in N(B)$ and $x\in dom(n)$, the equality $\alpha_n(x)=x$ implies that the germ of $\alpha_n$ at $x$ is trivial. According to \lemref{separation}, it suffices to show that $P(n)(x)\not=0$.  Given $n\in N(B)$ and $x\in dom(n)$, the states $x\circ P$ and $\alpha_n(x)\circ P$  are unitarily equivalent and  their transition probability (\cite{shu:pure}) is ${|P(n)(x)|^2\over n^*n(x)}$. Indeed, let $(H,\xi,\pi)$ be the GNS triple constructed from the state $x\circ P$. By construction,  $x\circ P$ is the state defined by the representation $\pi$ and the vector $\xi$. On the other hand, $\alpha_n(x)\circ P$ is the state of $A$ defined by $\pi$ and the vector $\eta=\pi(u)\xi$, where $u$ is the partial isometry of the polar decomposition $n=u|n|$ of $n$ in $A^{**}$. To show that, one checks the straightforward relation
$b(\alpha_n(x))=(\eta,\pi(b)\eta)$ for $b\in B$ and one uses the unique extension property. The transition probability can be computed by the formula $|(\xi,\eta)|^2={|P(n)(x)|^2\over n^*n(x)}$. If
$\alpha_n(x)=x$, the transition probability is 1. In particular, $P(n)(x)\not=0$. 
\end{proof}

\section{Examples of Cartan subalgebras in C$^*$-algebras}

\subsection{Crossed products by discrete groups}

In his pioneering work \cite{zm:croises} on crossed product C$^*$ and W$^*$-algebras by discrete groups, G.~Zeller-Meier gives the following necessary and sufficient condition (Proposition 4.14)  for $B$ to be maximal abelian in the reduced crossed product $C_r^*(\Gamma;B;\sigma)$, where $\Gamma$ is a discrete group acting by automorphisms on a commutative C$^*$-algebra $B$ and $\sigma$ is a 2-cocycle: the action of $\Gamma$ on $X=\hat B$ must be essentially free (some authors say topologically free), meaning that for all $s\in \Gamma\setminus\{e\}$, the set $X_s=\{x\in X: sx=x\}$ must have an empty interior in $X$. This amounts  to the groupoid $G=\Gamma\lsd X$ of the action being essentially principal. Proposition 2.4.7 of \cite{ren:approach} extends this result. Note that $G$ is principal if and only if the action is free, in the sense that
for all $s\in \Gamma\setminus\{e\}$, the set $X_s=\{x\in X: sx=x\}$ is empty. The particular case of the group $\Gamma={\bf Z}$ is well studied  (see for example \cite{tom:dynamics}) and we consider only this case below.

 Irrational rotations and minimal homeomorphisms of the Cantor space are examples of free actions. The C$^*$-algebras of these dynamical systems are well understood and completely classified. I owe to I.~Putnam the remark that the C$^*$-algebra of a Cantor minimal system contains uncountably many non-conjugate Cartan subalgebras (which are in fact diagonals in the sense of Kumjian). Indeed, according to 
\cite{gps:orbitequivalence}, such a C$^*$-algebra depends only, up to isomorphism, on the strong orbit equivalence class of the dynamical system; however, two minimal Cantor systems which are strongly orbit equivalent need not be flip conjugate (flip conjugacy amounts to groupoid isomorphism). More precisely, in any given strong orbit equivalence class, one can find homeomorphisms of arbitrary entropy. These will give the same C$^*$-algebra but the corresponding Cartan subalgebras will not be conjugate. 

On the other hand, two-sided Bernoulli shifts are examples of essentially free actions which are not free. They provide examples of Cartan subalgebras which do not have the extension property. In \cite{tom:dynamics}, J.~Tomiyama advocates the view that in relation with operator algebras, the notion of essentially free action, rather than that of free action,  is the counterpart for topological dynamical systems of the notion of free action for measurable dynamical systems. The comparison of \thmref{Cartan} and of the Feldman-Moore theorem completely supports this view.

\subsection{AF Cartan subalgebras in AF C$^*$-algebras}

Approximately finite dimensional (AF) C$^*$-algebras have privileged Cartan subalgebras. These are the maximal abelian subalgebras obtained by the diagonalization method of Str\u atil\u a and Voiculescu (\cite{sv:AF}). In that case, the twist is trivial and the whole information is contained in the Weyl groupoid. The groupoids which occur in that fashion are the AF equivalence relations. These are the equivalence relations $R$ on a totally disconnected locally compact Hausdorff space $X$ which are the union of an increasing sequence of proper equivalence relations $(R_n)$. The proper relations $R_n$ are endowed with the topology of $X\times X$ and $R$ is endowed with the inductive limit topology. As shown by Krieger in \cite{kri:dimension}, AF C$^*$-algebras and AF equivalence relations share the same complete invariant, namely the dimension group. One deduces that these privileged Cartan subalgebras, also called AF Cartan subalgebras, are conjugate by an automorphism of the ambient AF algebra. However, AF C$^*$-algebras may contain other Cartan subalgebras. An example of a Cartan subalgebra in an AF C$^*$-algebra without the unique extension property is given in \cite[III.1.17]{ren:approach}. A more striking example is given by B.~Blackadar in \cite{bla:CAR}. He constructs a diagonal in the CAR algebra whose spectrum is not totally disconnected. More precisely, he realizes the CAR algebra as the crossed product $C(X)\rsd\Gamma$ where $X={\bf S}^1\times\hbox{Cantor space}$ and $\Gamma$ is a locally finite group acting freely on $X$. Note that the groupoid $X\rsd\Gamma$ is also an AP equivalence relation, in the sense that it is the union of an increasing sequence of proper equivalence relations $(R_n)$.

\subsection{Cuntz-Krieger algebras and graph algebras}

The Cuntz algebra ${\mathcal O}_d$ is the prototype of a C$^*$-algebra which has a natural Cartan subalgebra without the unique extension property. By definition, ${\mathcal O}_d$ is the C$^*$-algebra generated by $d$ isometries $S_1,\ldots, S_d$ such that $\sum_{i=1}^d S_iS_i^*=1$. The Cartan subalgebra in question is the sub C$^*$-algebra ${\mathcal D}$ generated by the range projections of  the isometries $S_{i_1}\ldots S_{i_n}$. It can be checked directly that ${\mathcal D}$ is a Cartan subalgebra of ${\mathcal O}_d$; however, it is easier to  show first that $({\mathcal O}_d,{\mathcal D})$ is isomorphic to $(C^*(G), C(X))$, where $X=\{1,\ldots,d\}^{\bf N}$ and $G=G(X,T)$ is the groupoid associated to the one-sided shift $T:X\rightarrow X$ (see \cite{ren:approach, dea:groupoids, ren:cuntzlike}):
$$G=\{(x,m-n,y): x,y\in X, m,n\in{\bf N}, T^mx=T^ny\}.$$
This groupoid is not principal but it is essentially principal. In fact, the groupoid $G(X,T)$ associated to the local homeomorphism $T:X\rightarrow X$ is essentially principal if and only if $T$ is essentially free, meaning that for all pairs of distinct integers $(m,n)$,  the set $X_{m,n}=\{x\in X: T^mx=T^nx\}$ must have an empty interior in $X$. 

Condition (I) introduced by Cuntz and Krieger in their fundamental work \cite{ck:markov} ensures that the subalgebra ${\mathcal D}_A$ is a Cartan subalgebra of ${\mathcal O}_A$. Here, $A$ is a $d\times d$ matrix with entries in $\{0,1\}$ and non-zero rows and columns. The associated dynamical system is the one-sided subshift of finite type $(X_A,T_A)$; condition $(I)$ guarantees that this system is essentially free. In subsequent generalizations, in terms of infinite matrices in \cite{el:infinite} and in terms of graphs in \cite{kpr:graphs},  exit condition (L) replaces condition (I). On the topological dynamics side, it is a necessary and sufficient condition for the relevant groupoid to be essentially principal. On the C$^*$-algebraic side,  it is the condition which ensures that the natural diagonal subalgebra $\mathcal D$ is maximal abelian, hence a Cartan subalgebra. Moreover, it results from  \cite{kpr:graphs} that this subalgebra has the extension property if and only if the graph contains no loops. Condition (II) of \cite{cun:markov} or its generalization (K) in \cite{kprr:graphs} implies that each reduction of the groupoid to an invariant closed subset is essentially principal and therefore that the image of $\mathcal D$ in the corresponding quotient is still maximal abelian.

\subsection{Cartan subalgebras in continuous-trace C$^*$-algebras} Let us first observe that a Cartan subalgebra of a continuous-trace C$^*$-algebra necessarily has the unique extension property. The proof given in \cite[Th\'eor\`eme 3.2]{fac:remarques} for foliation C$^*$-algebras is easily adapted.

\begin{prop} Let $B$ be a Cartan subalgebra of a continuous-trace C$^*$-algebra $A$. Then $B$ has the unique extension property.
\end{prop}

\begin{proof} From the main theorem, we can assume that $(A,B)=(C^*_r(G,\Sigma),C_0(G^{(0)}))$, where $G$ is an \'etale essentially principal Hausdorff  groupoid and $\Sigma$ is a twist over $G$. Since $A$ is nuclear, we infer from 
\cite[6.2.14, 3.3.7]{dr:amenable} that $G$ is topologically amenable and from \cite[5.1.1]{dr:amenable} that all its isotropy subgroups are amenable.  Since $A$ is CCR, we infer from \cite[Section 5,]{cla:GCR} that $G^{(0)}/G$ injects continuously in $\hat A$ and that all the orbits of $G$ are closed (the presence of a twist does not affect this result nor its proof). Since $G$ is \'etale, these closed orbits are discrete. Now, each $h\in C_c(G^{(0)})$ belongs to the Pedersen ideal $K(A)$. Therefore, it defines a continuous function on $\hat A$ whose value at $[x]\in G^{(0)}/G$ is 
$${\overline h}[x]=\sum_{y\in [x]} h(y).$$ 
Suppose that $G(x)$ is not reduced to $\{x\}$. Then there exists an open neighborhood $V$ of $x$ such that $[x]\cap V=\{x\}$ and $[y]\cap V$ contains at least two elements for $y\not=x$. 
For $h\in C_c(G^{(0)})$ supported in $V$ and equal to 1 on a neighborhood of $x$, we would obtain ${\overline h}[x]=1$ and ${\overline h}[y]\ge 2$ for $y$ close to $x$, which contradicts the continuity of ${\overline h}$. Hence $G$ is principal and $B$ has the unique extension property.
\end{proof}
When a Cartan subalgebra $B$ of a continuous-trace C$^*$-algebra $A$ exists, the cohomology class $[\Sigma(B)]$ of  its twist is essentially the Dixmier-Douady invariant of $A$. Indeed, just as in the group case, the groupoid extension $\Sigma(B)$ defines an element of the cohomology group $H^2(G(B),{\bf T})$ (see \cite{tu:cohomology} for a complete account of groupoid cohomology). Since $G(B)$ is equivalent to $\hat B/G(B)=\hat A$, this can be viewed as an element of $H^2(\hat A,{\mathcal T})$, where $\mathcal T$ is the sheaf of germs of ${\bf T}$-valued continuous functions. Its identification with the Dixmier-Douady invariant is done in
 in \cite{kum:diagonals, ren:dual, rt:dd}. Moreover, a simple construction shows that every \u Cech cohomology class in $H^3(T,{\bf Z})$, where $T$ is a locally compact Hausdorff space, can be realized as the Dixmier-Douady invariant of a continuous-trace C$^*$-algebra of the above form $C^*(G,\Sigma)$.

However, Cartan subalgebras $B$ of a continuous-trace C$^*$-algebra $A$ do not always exist. It has been observed (see \cite [Remark 3.5.(iii)]{abg:extensions}) that there exist non-trivial $n$-homogeneous C$^*$-algebras which do not have a masa with the unique extension property. Therefore, these C$^*$-algebras do not have Cartan subalgebras. In \cite[Appendix]{kum:trace}, T.~Natsume gives an explicit example. Given a Hilbert bundle $H$ over a compact space $T$, let us denote by $A_H$ the continuous-trace C$^*$-algebra defined by $H$. Let $B$ be a Cartan subalgebra of $A_H$. The inclusion map gives a map $\hat B\rightarrow T$ which is a local homeomorphism and a surjection. If  $T$ is connected and simply connected, this is a trivial covering map and $B$ decomposes as a direct sum of summands isomorphic to $C(T)$. Therefore $H$ decomposes as a direct sum of line bundles. This is not always possible. For example there exists a vector bundle of rank 2 on the sphere ${\bf S}^4$ which cannot be decomposed into a direct sum of line bundles.

\subsection{Concluding remarks} Just as in the von Neumann setting, the notion of Cartan subalgebra in C$^*$-algebras provides a bridge between the theory of dynamical systems and the theory of operator algebras. Examples show the power of this notion, in particular to understand the structure of some C$^*$-algebras, but also its limits. This notion has to be modified if one wants to include the class of the C$^*$-algebras of non-Hausdorff essentially principal \'etale groupoids. In the case of continuous-trace C$^*$-algebras, we have seen that the twist attached to a Cartan subalgebra is connected with the Dixmier-Douady invariant. It would be interesting to investigate its C$^*$-algebraic significance in other situations.

\end{document}